\documentclass[11pt]{amsart}

\title[Mazur manifolds and corks with small shadow complexities]
{Mazur manifolds and corks with small shadow complexities}

\author[Hironobu Naoe]{Hironobu Naoe}
\address{Tohoku University, Sendai, 980-8578, Japan}
\email{sb3m22@math.tohoku.ac.jp}

\usepackage{amsfonts,amsmath,amssymb,amscd}
\usepackage{amsthm}
\usepackage{latexsym}
\usepackage{graphicx}
\usepackage{psfrag}
\usepackage{color}

\theoremstyle{plain}
\newtheorem{theorem}{Theorem}[section]
\newtheorem{lemma}[theorem]{Lemma}
\newtheorem{corollary}[theorem]{Corollary}

\newtheorem{remark}[theorem]{Remark}

\theoremstyle{definition}
\newtheorem{definition}[theorem]{Definition}
\newtheorem{example}[theorem]{Example}

\newcommand{\Z}{\mathbb{Z}}
\newcommand{\R}{\mathbb{R}}
\newcommand{\C}{\mathbb{C}}

\newcommand{\Int}{\mathrm{Int}}

\newcommand{\Nbd}{\mathrm{Nbd}}

\newcommand{\gl}{\mathrm{gl}}
\newcommand{\Mobius}{M\"{o}bius }

\makeatletter

\long\def\@makecaption#1#2{
  \small
  \vskip\abovecaptionskip
  \sbox\@tempboxa{#1. #2}
  \ifdim \wd\@tempboxa >\hsize
    #1. #2\par
  \else
    \global \@minipagefalse
    \hb@xt@\hsize{\hfil\box\@tempboxa\hfil}
  \fi
  \vskip\belowcaptionskip}

\makeatother
\begin{document}

\begin{abstract}
In this paper we find infinitely many Mazur type manifolds and corks 
with shadow complexity one 
among the 4-manifolds constructed from contractible special polyhedra 
having one true vertex by using the notion of Turaev's shadow. 
We also find such manifolds among 4-manifolds constructed from 
Bing's house. 
Our manifolds with shadow complexity one contain the Mazur manifolds 
$W^{\pm }(l,k)$ which were studied by Akbulut and Kirby. 
\end{abstract}

\maketitle

\section{Introduction}
\label{sec:Introduction}

The study of \textit{corks} is crucial to understand smooth 
structures on 4-manifolds due to the following theorem: 
\textit{For every exotic (i.e. homeomorphic but non-diffeomorphic) 
pair of simply connected closed 4-manifolds, 
one is obtained from the other by removing a contractible 
submanifold of codimension $0$ and gluing it via an involution 
on the boundary. Furthermore, the contractible 
submanifold and its complement can
always be compact Stein 4-manifolds.} 
The contractible 4-manifold has since been called a \textit{cork}. 
This theorem was first proved independently by 
Matveyev \cite{Mat96} and Curtis, Freedman, Hsiang and Stong \cite{CFHS96}, 
and strengthened by Akbulut and Matveyev \cite{AM98} afterward. 
The first cork was found by Akbulut in \cite{Akb91} among \textit{Mazur manifolds}. 
Here a \textit{Mazur manifold} is a contractible 
4-manifold which is not bounded by the 3-sphere 
and has a handle decomposition consisting of 
a single 0-handle, a single 1-handle and a single 2-handle. 
Akbulut and Yasui generalized the example in \cite{Akb91} and 
constructed many exotic pairs of 4-manifolds by using the corks 
$W_n, \overline{W}_n$ in \cite{AY08, AY12}. 

On the other hand,  Turaev introduced the notion of a \textit{shadow} for the 
purpose of study of quantum invariants of 3- and 4-manifolds in 1990s. 
A shadow is an almost-special polyhedron $P$ which is locally flat and properly 
embedded in a compact oriented 4-manifold $W$ with boundary 
and a strongly deformation retract of $W$. 
By the study of Turaev, $W$ is uniquely recovered from the shadow $P$ 
with a coloring assigned to each region of $P$. 
This operation is called Turaev's reconstruction. 
The coloring is a half-integer, called a \textit{gleam}. 
By this reconstruction, a shadow with gleam provides 
a differential structure of $W$ uniquely. 
We refer the reader to \cite{Cos06,Cos08}, 
in which Costantino studied Stein structures, 
Spin$^c$ structures and complex structures on 4-manifolds by using shadows. 
A classification of 4-manifolds with \textit{shadow complexity} zero is 
studied by Martelli in \cite{Mar11}. 
Here the shadow complexity is the minimal number of true vertices 
among all shadows of the 4-manifold. 

In this paper we study contractible 4-manifolds 
constructed from contractible shadows with complexity $1$ or $2$ 
and find infinitely many Mazur type manifolds and corks. 
We first focus on the shadows of complexity $1$. 
In \cite{Ike71,Ike72}, Ikeda classified 
acyclic special polyhedra with one true vertex 
and showed that there exist just two such polyhedra. 
One is called the abalone, shown in Figure \ref{abalone}, 
and the other is a polyhedron shown in Figure \ref{tildeA}.
We denote by $A$ the abalone, and by $\tilde{A}$ the other one. 
They are contractible since 
they are acyclic and simply connected. 
To find corks with shadow complexity $1$, 
we have only to study $A$ and $\tilde{A}$, 
by the above classification. 
Let $A(m,n)$ be the 
compact oriented 4-manifold obtained by 
Turaev's reconstruction from $A$ with 
gleams $\gl (e_1)=m, \gl (e_2)=n$ (see Figure \ref{abalone}), 
and let $\tilde{A}(m,n-\frac{1}{2})$ be the 
one constructed from $\tilde{A}$ with 
gleams $\gl (\tilde{e}_1)=m, \gl (\tilde{e}_2)=n-\frac{1}{2}$ (see Figure \ref{tildeA}). 
Note that the above $n$ and $m$ are integers. 
The main results of this paper in the case of complexity $1$ are the following. 

\begin{figure}
\begin{tabular}{cc}	
\begin{minipage}{0.42\hsize}
	\begin{center}
	\includegraphics[width=3.5cm]{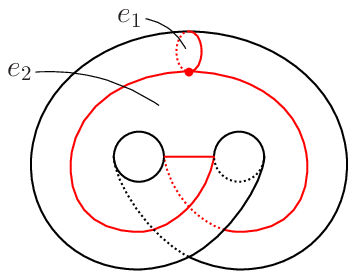}
	\caption{The abalone has only one true vertex and two regions. 
Let $e_1$ be the disk region in the upper part of the figure, and 
let $e_2$ be the other region. We can check that $e_2$ is a disk easily. 
}
	\label{abalone}
	\end{center}
\end{minipage}
\begin{minipage}{0.05\hsize}
\hspace{1ex}
\end{minipage}
\begin{minipage}{0.42\hsize}
	\begin{center}
	\includegraphics[width=5cm]{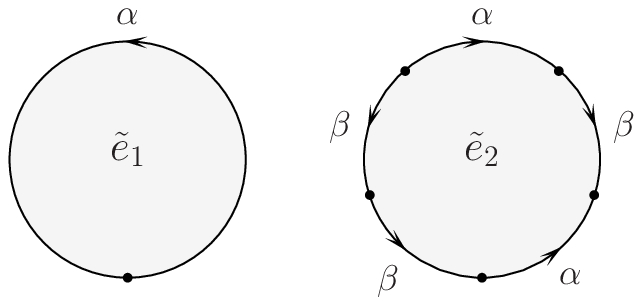}
	\caption{This polyhedron has only one true vertex and two disk regions. 
This can not be embedded in $\R ^3$. }
	\label{tildeA}
	\end{center}
\end{minipage}
\end{tabular}
\end{figure}

\begin{theorem}
\label{thm:A(m,n)}
If $m\ne 0$, $A(m,n)$ is Mazur type. 
Moreover, $A(m,n)$ and $A(m',n')$ are not homeomorphic unless $m = m'$. 
\end{theorem}

\begin{theorem}
\label{thm:tildeA}$ $
\begin{enumerate}
\item
If $m\ne 0$, $\tilde{A}(m,n-\frac{1}{2})$ is Mazur type. 
Moreover, $\tilde{A}(m,n-\frac{1}{2})$ and 
$\tilde{A}(m',n'-\frac{1}{2})$ are not homeomorphic unless $m = m'$. 
\item
The pair $(\tilde{A}(m,-\frac{3}{2}), \tilde{f}_m)$ is a cork 
if $m<0$. 
\end{enumerate}
\end{theorem}
Here $\tilde{f}_m$ is an involution on $\partial \tilde{A}(m,-\frac{3}{2})$, which 
will be defined in Section~\ref{sec:Proofs}. 

The following is a straightforward consequence of 
Theorem \ref{thm:A(m,n)}, \ref{thm:tildeA} and \cite{Ike71}. 
\begin{corollary}
\label{cor:shadow complexity}$ $
\begin{enumerate}
\item
There are no corks with shadow complexity $0$. 
\item
A cork with shadow complexity $1$ is either 
$A(m,n)$ or $\tilde{A}(m,n-\frac{1}{2})$. 
In particular, 
there are infinitely many corks with shadow complexity $1$ 
since there are infinitely many corks among $\tilde{A}(m,n-\frac{1}{2})$.
\end{enumerate}
\end{corollary}

\begin{figure}
	\begin{center}
	\includegraphics[width=4cm]{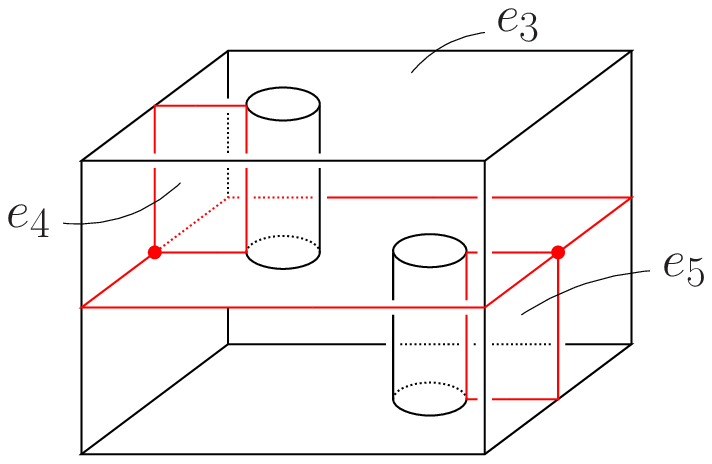}
	\caption{Bing's house has two true vertices and three disk regions. 
Let $e_4$ and $e_5$ be the disk regions which are partitions put first and second floor respectively, 
and let $e_3$ be the last region. 
We can check that $e_3$ is a disk easily. }
	\label{Bing}
	\end{center}
\end{figure}

Next we study Bing's house, which is a special polyhedron with two true vertices
as shown in Figure \ref{Bing}. 
Bing's house was introduced by Bing in \cite{Bin64}. 
We denote it by $B$ and 
let $B(l,m,n)$ be the 
compact oriented 4-manifold obtained by Turaev's reconstruction 
from $B$ with gleams $\gl (e_3)=l,\gl (e_4)=m,\gl (e_5)=n$, 
where $l,m,n$ are integers.  
For $B(l,m,n)$, we get the following. 
\begin{theorem}
\label{thm:B(l,m,n)} $ $
\begin{enumerate}
\item
If $|m|\geq 3$ and $|n|\geq 3$, 
then $B(l,m,n)$ is Mazur type. 
\item
The pair $(B(0,m,n),f_{(m,n)})$ is a cork if $m$ and $n$ are negative. 
\end{enumerate}
\end{theorem}
Here $f_{(m,n)}$ is an involution on $\partial B(0,m,n)$, which 
will be defined in Section \ref{sec:Proofs}. 

In Section \ref{sec:Preliminaries} we introduce the notions of shadows, 
Mazur manifolds, Stein surfaces 
and corks, and how to interpret from a Kirby diagram to a shadow. 
We give the proofs of our theorems in Section \ref{sec:Proofs}. 
Main tool in the proofs is Kirby calculus. 
To distinguish topological types of $A(m,n,)$'s and of $\tilde{A}(m,n-\frac{1}{2})$'s 
in Theorem \ref{thm:A(m,n)} and \ref{thm:tildeA} respectively, 
we compute the Casson invariants of their boundaries. 
To find corks among $\tilde{A}(m,n-\frac{1}{2})$ and $B(l,m,n)$, 
we check their Stein structures. 
The strategy is same as one used in \cite{AY08}. 

\subsection*{Acknowledgments. }
The author wishes to express his 
gratitude to his advisor, Masaharu Ishikawa, 
for encouragement and many helpful suggestions. 

\section{Preliminaries}
\label{sec:Preliminaries}

Throughout this paper, we work in smooth category unless otherwise mentioned. 

\begin{figure}
	\begin{center}
	\includegraphics[width=10cm]{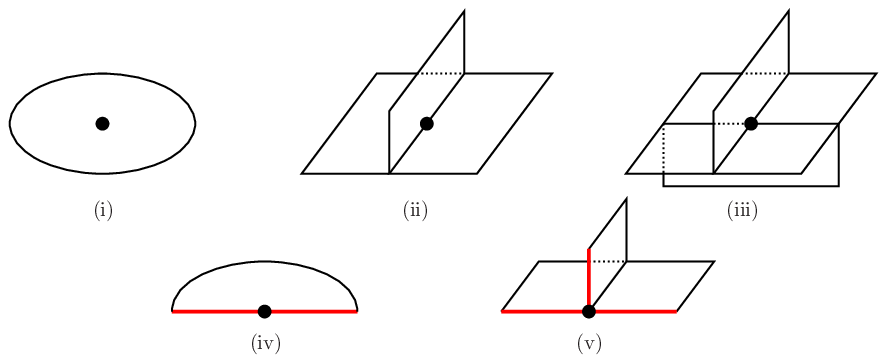}
	\caption{}
	\label{local model}
	\end{center}
\end{figure}

\subsection{Shadows}
\label{subsec:Shadows}

A compact topological space $P$ is called an {\it almost-special polyhedron} 
if each point of $P$ has a regular neighborhood which is homeomorphic to 
one of the five local models shown in Figure \ref{local model}. 
A point whose regular neighborhood is of type (iii) is called a {\it true vertex}. 
We denote the set of true vertices by $V(P)$. 
The {\it singular set} of $P$ is the set of points whose regular neighborhoods are 
of type (ii), (iii) or (v). We denote it by $Sing(P)$. 
The boundary $\partial P$ of $P$ is the set of points 
whose regular neighborhoods are of type (iv) or (v). 
Each component of $P\setminus Sing(P)$ is called a {\it region} of $P$. 
If a region $R$ contains points of type (iv) 
then $R$ is called a {\it boundary region}, 
and otherwise it is called an {\it internal region}. 
Each region is a surface. 
If each region of $P$ is homeomorphic to an open disk and 
any connected component of $Sing(P)$ contains at least one true vertex, 
then $P$ is said to be {\it special}. 
Each connected component of $Sing(P)\setminus V(P)$ is called a {\it triple line}. 

\begin{definition}
Let $W$ be a compact oriented 4-manifold 
and let $T$ be a (possibly empty) trivalent graph in the boundary $\partial W$ of $W$.
An almost-special polyhedron $P$ in $W$ is called a {\it shadow} of $(W,T)$ 
if the following hold:
\begin{itemize}
\item $W$ is collapsed onto $P$,
\item $P$ is locally flat in $W$, that is, for each point $p$ of $P$ there exists 
a local chart $(U,\phi)$ of $W$ around $p$ such that 
$\phi (U\cap P)\subset \R^3 \subset \R^4$ and
\item $P\cap \partial W=\partial P=T$. 
\end{itemize}
\end{definition}

It is well-known that any compact oriented 4-manifold having a handle decomposition without 
3- or 4-handles has a shadow \cite{Cos05}.

For a pair of topological spaces $X$ and $Y$ with $X\subset Y$, 
we denote a regular neighborhood of $X$ in $Y$ by $\Nbd (X;Y)$. 
Let $R$ be an internal region of an almost-special polyhedron $P$ 
and let $\bar{R}$ be a compact surface such that the interior of $\bar{R}$ is 
homeomorphic to $R$. 
The inclusion $i:R\to P$ can extend to a continuous map $\bar{i}:\bar{R}\to P$ 
such that $\bar{i}|_{\Int(\bar{R})}$ is injective and its image is the closure of $R$ in $P$. 
For each point $x\in \bar{i}(\partial \bar{R})$, 
we can see that, locally, two regions is attached to $\bar{R}$ along $\partial \bar{R}$. 
Under this identification, 
for each boundary component of $\bar{R}$, 
there exists an immersed annulus or a \Mobius band 
in $\Nbd (\bar{i}(\partial \bar{R});P)$. 
Let $N$ be the number of the \Mobius bands as above. 
For each internal region $R$, we choose a half integer $\gl (R)$ 
such that the following holds: 
\[
\gl (R)-\frac{1}{2}N \in \Z .
\]
We call $\gl (R)$ a {\it gleam} of $R$ 
and the correspondence $\gl$ a \textit{gleam} of $P$. 

An almost-special polyhedron $P$ endowed with a gleam 
is called an {\it integer shadowed polyhedron}, 
and denoted by $(P,\gl )$ (or simply $P$) .

Turaev showed that 
there exists a canonical mapping associating to 
an integer shadowed polyhedron $(P,\gl )$ 
a compact oriented smooth 4-manifold, denoted by $M_P$, in \cite{Tur94}. 
This is called \textit{Turaev's reconstruction}. 
The method of this reconstruction is analogous to 
a process of attaching handles. 
Conversely, if a 4-manifold $M$ has a shadow $P$ then 
$P$ can be equipped with the canonical gleam $\gl$ such that 
the 4-manifold $M_P$ reconstructed from $(P,\gl )$ is diffeomorphic to $M$. 

\subsection{Interpretation from a Kirby diagram to a shadow}
\label{subsec:Interpretation from a Kirby diagram to a shadow}
Now we introduce a method to obtain a shadow of a 4-manifold 
which is given by a Kirby diagram. 
We follow the method in \cite[Chapter IX. 3.2.]{Tur94} and \cite{Cos05}. 

Let $\mathcal{D} = 
(\bigsqcup _{i=1}^{k} L_i^1)\amalg (\bigsqcup _{j=1}^{l} L_j^2)
\subset S^3$ be a Kirby diagram, where $L_i^1$ ($i=1,\ldots ,k$) 
is a dotted circle of a 1-handle and $L_j^2$ ($j=1,\ldots ,l$) 
is a attaching circle of a 2-handle 
with framing coefficient $n_j$. 
We arbitrarily take an almost-special polyhedron $Q$ 
in $S^3 \setminus (\bigsqcup _{i=1}^{k} L_i^1)$ 
such that $S^3 \setminus (\bigsqcup _{i=1}^{k} L_i^1)$ collapses onto $Q$. 
By isotopy in $S^3 \setminus (\bigsqcup _{i=1}^{k} L_i^1)$, 
we project $\bigsqcup _{j=1}^{l} L_j^2$ to $Q\setminus V(Q)$ 
such that each crossing point is a double point and not on $Sing(Q)$. 
We denote the image by $C=\bigcup _{j=1}^{l} C_j$. 
We then assign an over/under information to each crossing point 
such that the link restored from $C$ according to the over/under information 
is isotopic to $\bigsqcup _{j=1}^{l} L_j^2$. 
\begin{figure}
\begin{tabular}{ccc}
\begin{minipage}{0.64\hsize}
\begin{center}
	\includegraphics[height=1.8cm]{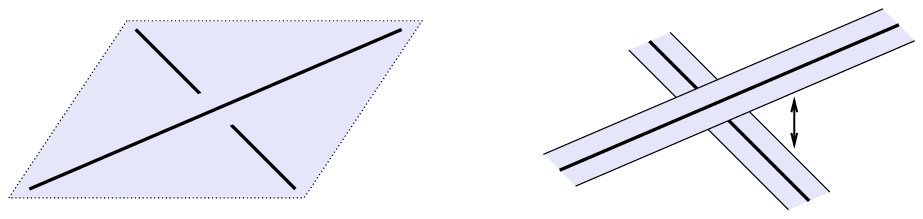}
	\caption{The first figure indicates a crossing point of a 
projection of a link component on an almost-special polyhedron. 
Then we can take the framing with respect to this almost-special 
polyhedron as in the second figure. }
	\label{framing w.r.t. Q}
\end{center}
\end{minipage}
\begin{minipage}{0.07\hsize}
\hspace{1ex}
\end{minipage}
\begin{minipage}{0.24\hsize}
\begin{center}
\vspace{-4mm}
	\includegraphics[height=2cm]{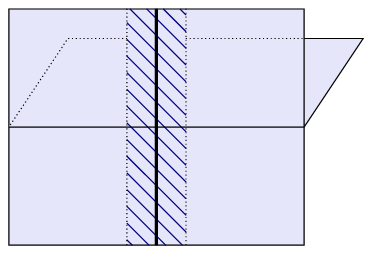}
	\caption{The framing is the hatched band. }
	\label{framing w.r.t. Q2}
\end{center}
\end{minipage}
\end{tabular}
\end{figure}

\begin{definition}
\label{def:framing with respect to Q. }$ $
\begin{enumerate}
\item
We call the \textit{framing with respect to} $Q$ \textit{of} $C_j$ 
an embedded annulus or \Mobius band in $S^3$ obtained 
by taking a small regular neighborhood of $C$ in $Q$ and 
splitting it at each crossing point 
according to the over/under information as in Figure \ref{framing w.r.t. Q}. 
Here if $C$ runs over a triple line, 
we cut off the part which $C$ does not lie 
from $\Nbd (C;Q)$
as in Figure \ref{framing w.r.t. Q2}. 
\item
Let $F_j$ be the framing with respect to $Q$ of $C_j$. 
If $F_j$ is an annulus, 
\textit{the twist number of $F_j$} is defined by the linking number of 
the link $S^1 \amalg S^1 = \partial F_j \subset S^3$ whose orientations are 
chosen to be parallel. 
If $F_j$ is a \Mobius band, 
\textit{the twist number of $F_j$} is defined by 
the half of the linking number of the link 
$S^1 \amalg S^1 = \partial F_j \amalg \text{(a core of $F_j$)} \subset S^3$ 
whose orientations are chosen to be parallel. 
We denote the twist number of $F_j$ by $tw(C_j)$. 
\end{enumerate}
\end{definition}

\begin{figure}
\begin{center}
	\includegraphics[width=8cm]{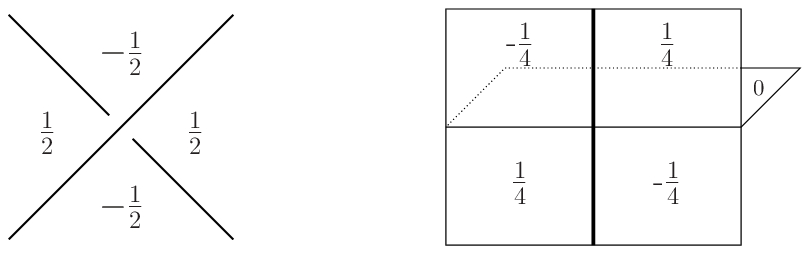}
	\caption{
The local contributions to gleams: 
The left figure indicates the local contributions at s self-crossing point of $C$, 
and the right one indicates those at a crossing point of $C$ and $Sing(Q)$. 
}
	\label{the contribution}
\end{center}
\end{figure}

Let $P$ be the almost-special polyhedron obtained 
from $Q$ by attaching a disk $D_j$ to each curve $C_j$ along its boundary. 
Note that $P$ is not necessarily embedded in $S^3$. 
We define the gleam $\gl$ of $P$ in the following way. 
\begin{itemize}
\item
Let $R$ be an internal region of $P$ in the subdivision of $Q$ by $C\cup Sing(Q)$. 
To four separated regions in a small regular neighborhood of 
each self-crossing point of $C$ or each crossing point of $C$ and the triple line of $Q$, 
we assign rational numbers as shown in Figure \ref{the contribution} 
(cf. \cite{Cos05, Tur94}). 
We define $\gl(R)$ by the sum of these local contributions 
for all crossing points of $C\cup Sing(Q)$ to which $R$ is adjacent. 
\item
The gleam of the region $\Int (D_j)$ is defined by $n_j - tw(C_j)$. 
\end{itemize}
Thus we get an integer shadowed polyhedron $(P,\gl )$. 

\begin{lemma}
\label{lem:Kirby to shadow}
The 4-manifold reconstructed from $(P,\gl )$ is 
diffeomorphic to the 4-manifold given by the Kirby diagram $\mathcal{D}$. 
\end{lemma}

\begin{proof}
We only give a sketchy proof of this lemma.
For details we refer the reader to \cite{Tur94} and \cite{Cos05}. 
We only verify that 
the gleam of the region $R_j$, which is the interior of $D_j$, 
is compatible with attaching the 2-handle. 
The framing of the 2-handle corresponding to $L_j^2$ is represented by 
a knot $\hat{L}_j^2$ parallel to $L_j^2$. 
Let $B_j$ be an annulus whose boundaries are $L_j^2$ and $\hat{L}_j^2$. 
If $B_j$ can embed in the framing $F_j$ with respect to $Q$ of $C_j$ 
by isotopy sending $L_j^2$ to $C_j$, 
the gleam of $R_j$ is $0$ by \cite{Tur94}. 
Since the framing coefficient $n_j$ is defined by $lk(L_j^2, \hat{L}_j^2)$ and 
the gleam of $R_j$ increases by the number 
of the twists of $B_j$ with respect to $F_j$, 
we conclude that the gleam of $R_j$ is given by $n_j - tw(C_j)$. 
\end{proof}

Note that, gthe gleam of a boundary regionh can be defined as above but 
it does not affect the reconstruction. 
Therefore, if the boundary region of $P$ can collapse to the singular set $Sing(P)$, 
the resulting polyhedron is also a shadow of $M_P$. 
Two or more regions may be connected by this collapsing. 
In this case the gleam of the new region is given by 
the sum of the original gleams before connecting. 

\subsection{Mazur manifolds and Akbulut's corks}
\label{subsec:Mazur manifolds and Akbulut's corks}
\begin{definition}
A compact contractible 4-manifold which is not a 4-ball is called a 
{\it Mazur manifold} if it is obtained by attaching a 2-handle to $D^3\times S^1$. 
If a Mazur manifold is not bounded by the 3-sphere then it is said to be {\it Mazur type}.
\end{definition}

\begin{figure}
\begin{center}
	\includegraphics[height=3.1cm]{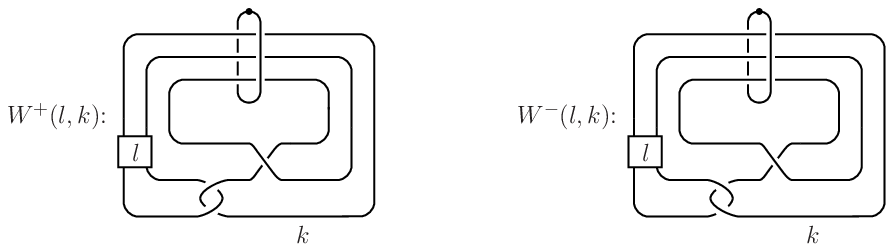}
	\caption{}
	\label{W(l,k)}
\end{center}
\end{figure}

In \cite{Maz61}, Mazur introduced Mazur manifolds $W^{\pm}(l,k)$, 
described in Figure \ref{W(l,k)}, 
which were studied by Akbulut and Kirby in \cite{AK79}. 
Remark that any compact contractible 4-manifold is bounded by a homology 3-sphere. 
In particular, if a shadow is contractible then its 4-manifold is also. 
Therefore such a 4-manifold always becomes a candidate of a Mazur manifold. 

Now we introduce the definition of a \textit{cork}, which was first found by Akbulut 
among Mazur manifolds in \cite{Akb91}. 

\begin{definition}
Let $C$ be a contractible Stein domain and 
let $\tau :\partial C\to \partial C$ be an involution on the boundary of $C$.
The pair $(C,\tau )$ is called a {\it cork} if 
$\tau$ extends to a self-homeomorphism on $C$, 
but can not extend to any self-diffeomorphism on $C$.
\end{definition}

Here a {\it Stein manifold} is a complex manifold $X$ such that 
$X$ can be embedded into $\C ^N$ by a proper holomorphic map. 
A compact 4-manifold $W$ with boundary is called a {\it Stein domain} if 
there exists a Stein 4-manifold $X$ with 
a plurisubharmonic function $f:X \to [0,\infty)$ 
and a regular value $a$ of $f$ such that $f([0,a])\cong W$. 

\begin{example}
\label{ex:cork}
Akbulut and Yasui constructed many corks $(W_n,f_n)$ ($n\geq 1$) in \cite{AY08}. 
Note that $W_1$ is just $W^{-}(0,0)$. 
They defined $f_n:\partial W_n\to \partial W_n$ by the involution 
obtained by first surgering $S^1 \times B^3$ to $B^2 \times S^2$ 
in the interior of $W_n$, 
and then surgering the other embedded $B^2 \times S^2$ back to $S^1 \times B^3$. 
We notice that the Kirby diagrams of $W_n$ in \cite{AY08} is a symmetric link, and 
the involution can be done by replacing the dot and g$0$h in the diagram. 
\end{example}

\section{Proofs}
\label{sec:Proofs}
\begin{figure}
\begin{tabular}{ccc}
\begin{minipage}{0.42\hsize}
	\begin{center}
	\includegraphics[width=4.25cm]{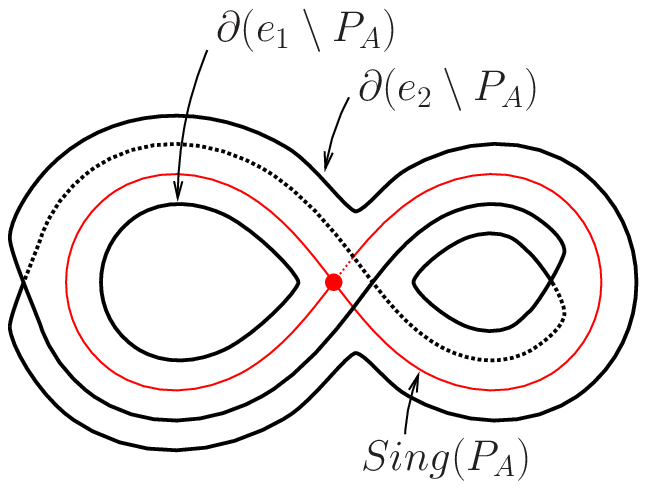}
	\caption{The almost-special polyhedron $P_A$.}
	\label{Nbd(abalone)}
	\end{center}
\end{minipage}
\begin{minipage}{0.05\hsize}
\hspace{1em}
\end{minipage}
\begin{minipage}{0.42\hsize}
	\begin{center}
	\includegraphics[width=3.4cm]{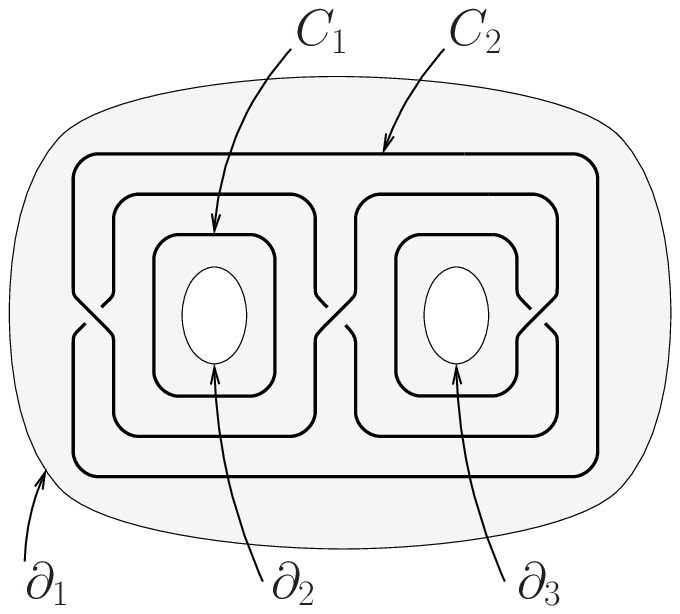}
	\caption{The pair of pants $Q_A$ with 
immersed curves $C_1$ and $C_2$.}
	\label{pants}
	\end{center}
\end{minipage}
\end{tabular}
\begin{center}
	\includegraphics[width=12.5cm]{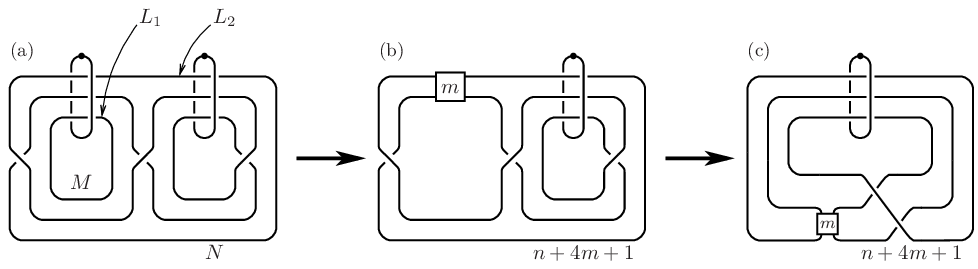}
	\caption{}
	\label{reconstruct_abalone}
\end{center}
\end{figure}

This section separates into two parts. 
In the former part we give the proofs of Theorem \ref{thm:A(m,n)} and \ref{thm:tildeA}, 
and we give the proof of Theorem \ref{thm:B(l,m,n)} in the latter part. 

\subsection{The case of $1$ true vertex}
\begin{lemma}
\label{lem:Kirby diag A(m,n)}
The 4-manifold $A(m,n)$ is represented by the Kirby diagram 
shown in Figure \ref{reconstruct_abalone} (c). 
\end{lemma}
\begin{proof}
Let $P_A$ be a small regular neighborhood of the singular set of $A$. 
It is easy to see that $P_A$ is an almost-special polyhedron shown in Figure \ref{Nbd(abalone)},
that is, $A$ is obtained from $P_A$ by attaching a disk along each of 
the boundary components $\partial (e_1\setminus P_A)$ and 
$\partial (e_2\setminus P_A)$ shown in the figure. 

Now let us consider the pair of pants $Q_A$ shown in Figure \ref{pants}, 
where $\partial _1, \partial _2$ and $\partial _3$ are 
its boundary components, 
and $C_1$ and $C_2$ are immersed curves on $Q_A$ 
equipped with over/under information at each double point. 
Let $Q'_A$ be the new almost-special polyhedron obtained from $Q_A$ 
by attaching two disks $D_1$ and $D_2$ to $C_1$ and $C_2$ along 
the boundaries respectively. 
Note that $Q'_A$ can not be embedded in $S ^3$. 
It is easy to check that the almost-special polyhedron obtained from $Q'_A$ by collapsing 
the three boundary regions corresponding to 
$\partial _1, \partial _2$ and $\partial _3$ is $A$. 
Note that $D_i\subset e_i$ for $i=1,2$ under this identification. 

We next consider the Kirby diagram shown in Figure \ref{reconstruct_abalone} (a). 
In the figure, $L_1$ and $L_2$ are attaching circles of 2-handles, 
whose framing coefficients are represented by 
$M$ and $N$ as in the figure respectively. 
We can see that $Q_A$ can be embedded in the Kirby diagram 
in Figure \ref{reconstruct_abalone} (a) such that 
$S^3\setminus (\text{dotted circles})$ collapses to $Q_A$ and 
$L_1$ and $L_2$ are projected onto $C_1$ and $C_2$ respectively. 
We can also see that the 4-manifold given by 
Figure \ref{reconstruct_abalone} (a) has a shadow $Q'_A$, and then also $A$. 

What is left is to compute the relation between $M, N$ and gleams 
$m=\gl (e_1), n=\gl (e_2)$. 
Let $F_i$ ($i=1,2$) be the framing with respect to $Q_A$ of $C_i$. 
Both $F_1$ and $F_2$ are annuli and 
their twist numbers are $tw(C_1) = 0$ and $tw(C_2) = 1$. 
Therefore, $\gl (D_1) = M - 0$, $\gl (D_2) = N - 1$ 
by Lemma \ref{lem:Kirby to shadow}. 
Considering the changes of the gleams by collapsing of the boundary regions, 
we get 
\begin{align*}
m &= \gl (e_1) = M + (\frac{1}{2}-\frac{1}{2}), \\
n &= \gl (e_2) = (N -1) + (4\cdot \frac{1}{2} - 4\cdot \frac{1}{2}). 
\end{align*}
Hence $M=m$ and $N = n+1$. 
Thus we get the Kirby diagram of $A(m,n)$ in Figure \ref{reconstruct_abalone} (a). 
Since the pair of the left 1-handle and the 2-handle corresponding to $L_1$
is a canceling pair, 
we can erase this pair and get Figure \ref{reconstruct_abalone} (b). 
By isotopy, we get the Kirby diagram of $A(m,n)$ as in 
Figure \ref{reconstruct_abalone} (c). 
\end{proof}

\begin{figure}
\begin{tabular}{ccc}
\begin{minipage}{0.6\hsize}
	\begin{center}
	\includegraphics[width=7.6cm]{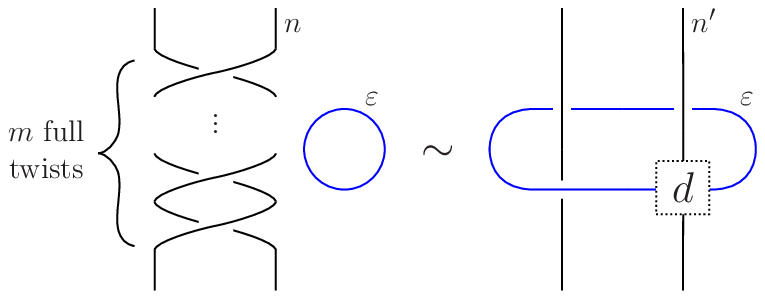}
	\caption{}
	\label{lem}
	\end{center}
\end{minipage}
\begin{minipage}{0.15\hsize}
\hspace{1ex}
\end{minipage}
\begin{minipage}{0.20\hsize}
\vspace{1.8ex}
	\begin{center}
	\includegraphics[width=2cm]{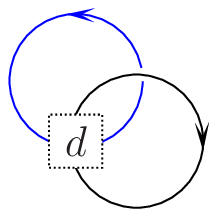}
	\caption{}
	\label{lem2}
	\end{center}
\end{minipage}
\end{tabular}
\end{figure}

\begin{lemma}
\label{lem:Kirby}
There exists a sequence of Kirby moves from the left to the right in Figure \ref{lem} for any $m$,
where the tangle $d$ satisfies the following properties:
\begin{enumerate}
\item
the linking number of the link shown in Figure \ref{lem2} is $|m|$ and
\item
the link component with framing coefficient $\varepsilon $ 
in Figure \ref{lem} is an unknot, 
\end{enumerate}
where $\varepsilon = -1$ if $m>0$ and $\varepsilon = +1$ if $m<0$.
\end{lemma}

\begin{figure}
	\begin{center}
	\includegraphics[width=10cm]{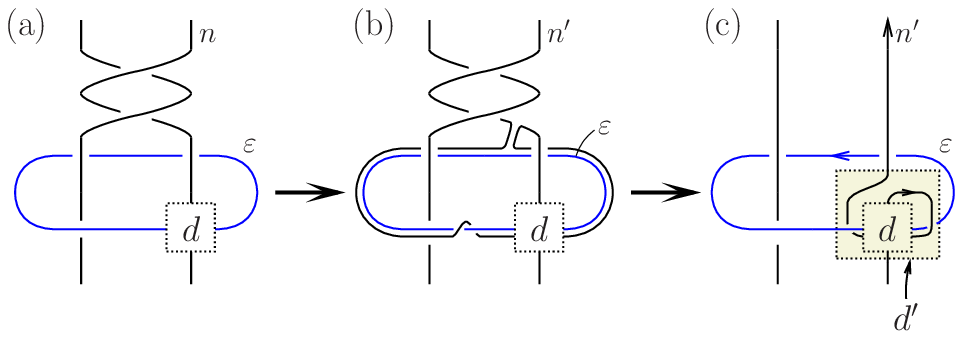}
	\vspace{-1em}
	\caption{}
	\label{lem3}
	\end{center}
	\begin{center}
	\includegraphics[width=10cm]{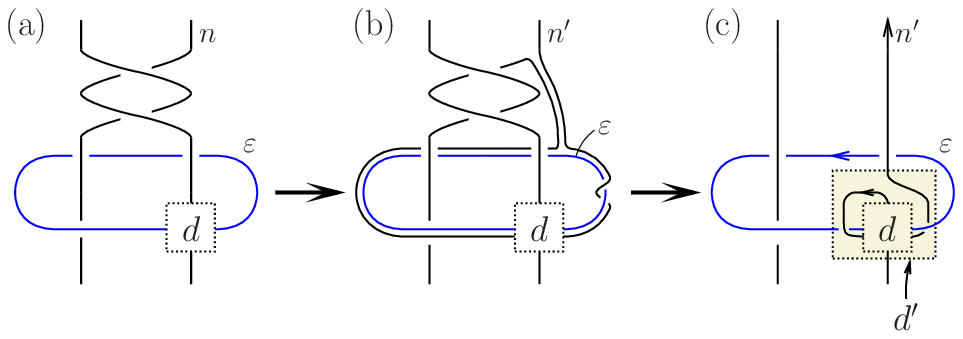}
	\vspace{-1em}
	\caption{}
	\label{lem4}
	\end{center}
\end{figure}
\begin{proof}
We first prove for the case $m>0$ by induction on $m$. 
It is easy for $m=1$. 

See Figure \ref{lem3}. 
We assume that 
this lemma holds 
for $m$ and prove it for $m+1$. 
We slide the 2-handle and get Figure \ref{lem3} (b). 
The isotopy gives Figure \ref{lem3} (c). 
Since the link component with framing coefficient $\varepsilon $ is not moved, 
it is still an unknot. 
Let $d'$ be the tangle shown in Figure \ref{lem3} (c). 
In Figure \ref{lem3} (b) and (c), the link component with framing coefficient $n'$ 
intersects the vertical sides of the tangle $d$ 
and is running parallel to the other strand
with framing coefficient $\varepsilon$ inside of $d$. 
Note that these two parallel strands in $d$ have
no crossing because of the condition (2). 
Hence there is no change of crossing number in $d$. 
On the other hand, out of $d$ in $d'$, we can see two positive crossings. 
Then computing the linking number of the link 
obtained by setting $d'$ instead of the tangle $d$ in Figure \ref{lem2}, 
we get $m+1$.

The proof for the case $m<0$ is similar. 
See Figure \ref{lem4}. 
\end{proof}


\begin{proof}[Proof of Theorem \ref{thm:A(m,n)}]

The 4-manifold $A(m,n)$ is contractible since $A$ is contractible.
By the Kirby diagram in Figure \ref{reconstruct_abalone} (c), 
$A(m,n)$ satisfies the condition of the handle decomposition 
for a Mazur manifold. 
We will compute the Casson invariant of the boundary of $A(m,n)$ to 
verify whether $A(m,n)$ is Mazur type. 

\begin{figure}
\begin{center}
	\includegraphics[width=8cm]{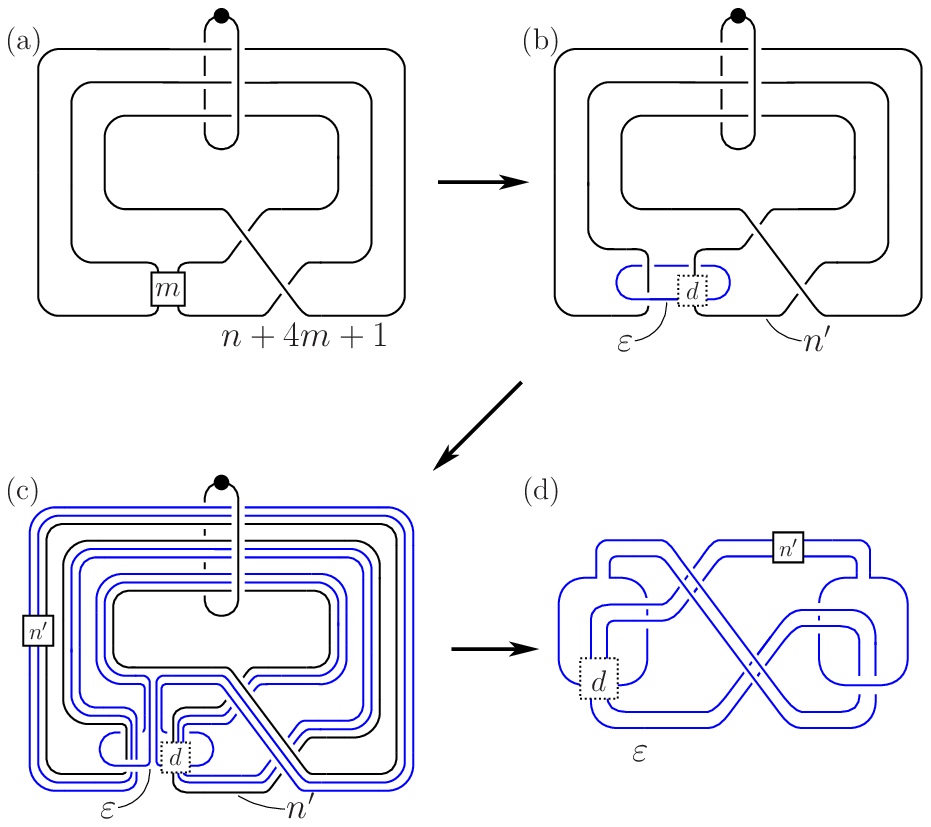}
	\caption{}
	\label{surg_diag_A(m,n)}
\end{center}
\end{figure}

Let us blow-up and apply Lemma \ref{lem:Kirby} to the diagram shown in 
Figure \ref{surg_diag_A(m,n)} (a), and we get Figure \ref{surg_diag_A(m,n)} (b). 
Sliding the 2-handle with framing coefficient $\varepsilon$ another 2 times, 
we get Figure \ref{surg_diag_A(m,n)} (c). 
We erase the canceling pair and get Figure \ref{surg_diag_A(m,n)} (d). 
Now we focus on the boundary of $A(m,n)$ and regard
Figure \ref{lem:Kirby} (d) as a surgery diagram of $\partial A(m,n)$.
Note that the two strands in the tangle $d$, 
intersecting two horizontal sides of $d$, are parallel. 
Moreover, the strand in the tangle $d$, 
intersecting two vertical sides of $d$ has no self-intersection. 
Hence the knot in Figure \ref{surg_diag_A(m,n)} (d) is a ribbon knot. 
We denote this knot by $K_{(m,n)}$. 
Next we compute the Alexander polynomial of $K_{(m,n)}$ by using the method 
in \cite{Ter59}, in which Terasaka computed the Alexander polynomials of ribbon knots. 
We get 
\begin{align*}
\mathit{\Delta}_{K_{(m,n)}}(t) 
= t^{|m|+1}-t^{|m|}-t+3-t^{-1}-t^{-|m|}+t^{-|m|-1} . 
\end{align*}
Then the Casson invariant of $\partial A(m,n)$ 
can be computed by using the surgery formula for the Casson invariant as follows:
\begin{align*}
\lambda (\partial A(m,n)) 
&= \lambda (S^3 + \frac{1}{\varepsilon}\cdot K_{(m,n)})\\
&= \lambda (S^3) + \frac{\varepsilon}{2}\mathit{\Delta}_{K_{(m,n)}}''(1)\\
&= 0 + \frac{\varepsilon}{2}\cdot 4|m|\\
&= -2m.
\end{align*}
Hence $\partial A(m,n)$ is not homeomorphic to $S^3$ for $m\ne 0$, and
$A(m,n)$'s are mutually not homeomorphic for different $m$. 
\end{proof}

Next we study the compact oriented 4-manifold $\tilde{A}(m,n-\frac{1}{2})$ 
whose shadow is $\tilde{A}$. 
We first describe $\tilde{A}(m,n-\frac{1}{2})$ by the Kirby diagram to prove 
Theorem \ref{thm:tildeA}. 

\begin{lemma}
\label{lem:Kirby diag tildeA}
The 4-manifold $\tilde{A}(m,n-\frac{1}{2})$ is represented by the Kirby diagram
shown in Figure \ref{reconstruct_tildeA} (c). 
\end{lemma}
\begin{proof}
The proof is similar to Lemma \ref{lem:Kirby diag A(m,n)}. 
Let $P_{\tilde{A}}$ be a small regular neighborhood of the singular set 
of $\tilde{A}$ shown in Figure \ref{Nbd(tildeA)}. 
It is easy to check that $\tilde{A}$ is homeomorphic to 
$P_{\tilde{A}}$ attached two disks along the boundary components 
$\partial (\tilde{e}_1\setminus P_{\tilde{A}})$ and 
$\partial (\tilde{e}_2\setminus P_{\tilde{A}})$ 
shown in the figure. 

Now let us consider the torus $Q_{\tilde{A}}$ with one boundary component 
as shown in Figure \ref{punctured_torus}, 
where $\tilde{\partial}$ is its boundary component 
and $\tilde{C}_1$ and $\tilde{C}_2$ are the immersed curves on $Q_{\tilde{A}}$ 
equipped with over/under information at each double point. 

Let $Q'_{\tilde{A}}$ be the almost-special polyhedron obtained from $Q_{\tilde{A}}$ 
by attaching two disks $\tilde{D}_1$ and $\tilde{D}_2$ 
to $\tilde{C}_1$ and $\tilde{C}_2$ along the boundaries respectively. 
It is easy to check that the almost-special polyhedron obtained from $Q'_{\tilde{A}}$ by collapsing 
the boundary region is $\tilde{A}$. 
Note that $\tilde{D}_i\subset \tilde{e}_i$ for $i=1,2$ under this identification. 

Next we consider the Kirby diagram shown in Figure \ref{reconstruct_tildeA} (a). 
In the figure, $\tilde{L}_1$ and $\tilde{L}_2$ are attaching circles of 2-handles, 
whose framing coefficients are represented by 
$M$ and $N$ as in the figure respectively. 
We can see that $Q_{\tilde{A}}$ can be embedded in the Kirby diagram 
in Figure \ref{reconstruct_tildeA} (a) such that 
$S^3\setminus (\text{dotted circles})$ collapses to $Q_{\tilde{A}}$ and 
$\tilde{L}_1$ and $\tilde{L}_2$ are projected onto 
$\tilde{C}_1$ and $\tilde{C}_2$ respectively. 
We can also see that the 4-manifold given by 
Figure \ref{reconstruct_tildeA} (a) has a shadow $Q'_{\tilde{A}}$, 
and then also $\tilde{A}$. 

Let $\tilde{F}_i$ ($i=1,2$) be the framing with respect to $Q_{\tilde{A}}$ of $\tilde{C}_i$. 
Both $\tilde{F}_1$ and $\tilde{F}_2$ are annuli, 
and we get the twist numbers $tw(\tilde{C}_1) = 0$ and $tw(\tilde{C}_2) = 1$. 
Therefore, $\gl (\tilde{D}_1) = M - 0$, $\gl (\tilde{D}_2) = N - 1$ 
by Lemma \ref{lem:Kirby to shadow}. 
Considering the changes of the gleams by collapsing of the boundary regions, 
we get 
\begin{align*}
m &= \gl (\tilde{e}_1) = M + 0, \\
n-\frac{1}{2} &= \gl (\tilde{e}_2) = (N - 1) - \frac{1}{2}. \\
\end{align*}
Hence $M=m$ and $N = n+1$. 
Thus we get the Kirby diagram of $\tilde{A}(m,n-\frac{1}{2})$. 
We erase the canceling pair and get Figure \ref{reconstruct_tildeA} (b). 
By isotopy, we get Figure \ref{reconstruct_tildeA} (c). 
\end{proof}
\begin{figure}
\begin{tabular}{ccc}
\begin{minipage}{0.41\hsize}
	\begin{center}
	\includegraphics[height=3.5cm]{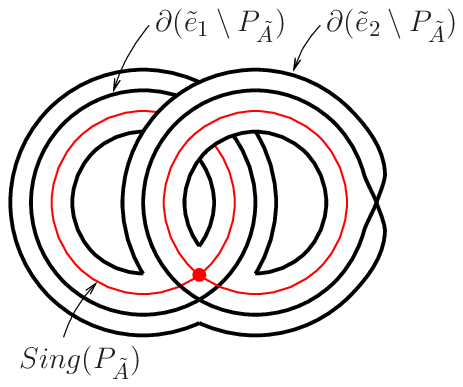}
	\caption{The almost-special polyhedron $P_{\tilde{A}}$.}
	\label{Nbd(tildeA)}
	\end{center}
\end{minipage}
\begin{minipage}{0.05\hsize}
\hspace{1em}
\end{minipage}
\begin{minipage}{0.42\hsize}
\vspace*{3ex}
	\begin{center}
	\includegraphics[height=3.5cm]{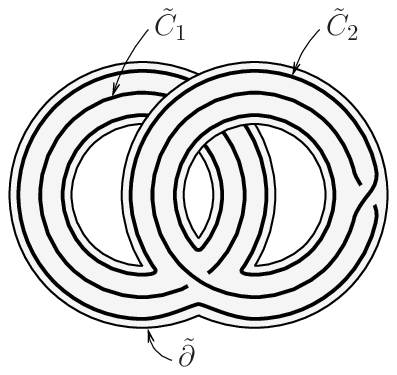}
	\caption{The torus $Q_{\tilde{A}}$ with one boundary component 
and immersed curves $\tilde{C}_1$ and $\tilde{C}_2$.}
	\label{punctured_torus}
	\end{center}
\end{minipage}
\end{tabular}
\begin{center}
	\includegraphics[width=1\hsize]{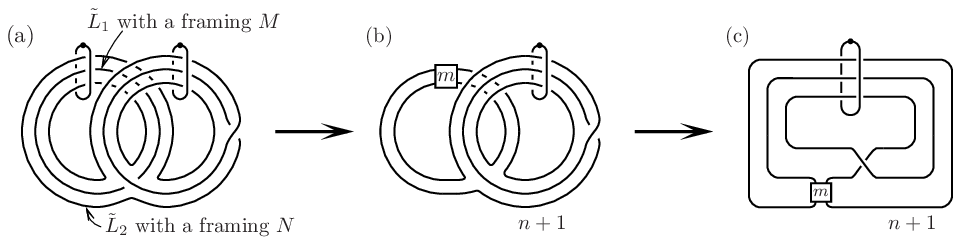}
	\caption{}
	\label{reconstruct_tildeA}
\end{center}
\end{figure}

\begin{remark}
\label{rem:A=tildeA}
It is easy to check that 
$A(\pm 1,n)\cong \tilde{A}(\mp 1,n\pm \frac{7}{2})$ by isotopies in their Kirby diagrams. 
We don't know whether or not there are other diffeomorphisms between 
$A(m,n)$ and $\tilde{A}(m',n'-\frac{1}{2})$. 
\end{remark}

We now define 
\begin{align*}
\tilde{f}_m:\partial \tilde{A}(m,-\frac{3}{2})\to \partial \tilde{A}(m,-\frac{3}{2})
\end{align*}
by the involution obtained by first surgering $S^1 \times B^3$ to $B^2 \times S^2$ 
in the interior of $\tilde{A}(m,-\frac{3}{2})$, 
then surgering the other embedded $B^2 \times S^2$ back to $S^1 \times B^3$. 
We notice the Kirby diagram of $\tilde{A}(m,n-\frac{1}{2})$ in Figure \ref{reconstruct_tildeA} (c) 
is a symmetric link, and the involution can be done by replacing the dot and g$0$h in the diagram. 

\begin{figure}
\begin{center}
	\includegraphics[width=8cm]{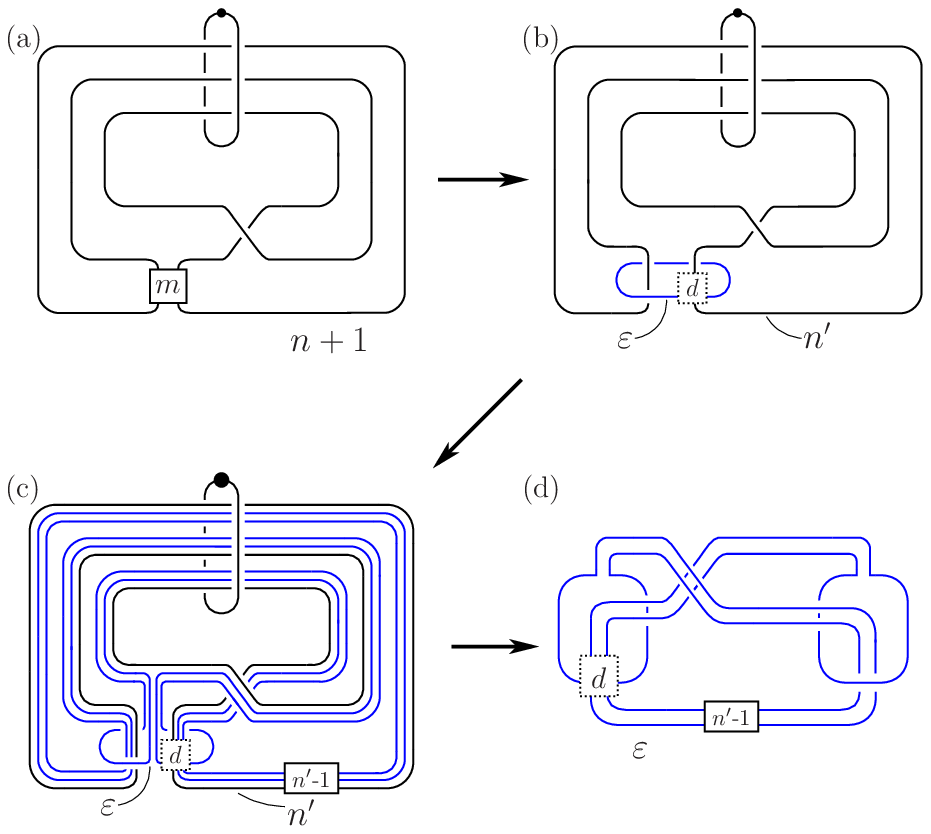}
	\caption{}
	\label{surg_diag_tildeA}
\end{center}
\begin{center}
	\includegraphics[width=6cm]{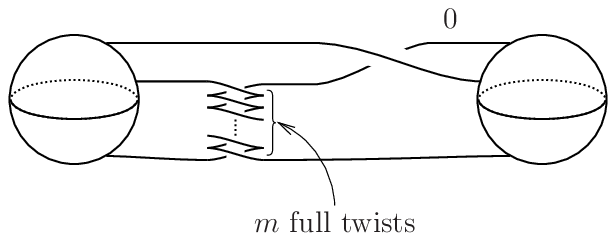}
	\caption{}
	\label{Legendred_tildeA}
\end{center}
\end{figure}

\begin{proof}[Proof of Theorem \ref{thm:tildeA}]
\noindent (1)
The proof is similar to Theorem \ref{thm:A(m,n)}. 
First the 4-manifold $\tilde{A}(m,n-\frac{1}{2})$ is 
contractible since $\tilde{A}$ is contractible. 
By the Kirby diagram, 
$\tilde{A}(m,n-\frac{1}{2})$ has a handle decomposition satisfying the 
condition for a Mazur manifold. 

We next compute the Casson invariant of the boundary of $\tilde{A}(m,n-\frac{1}{2})$. 
Let us blow-up and apply Lemma \ref{lem:Kirby} to the diagram as shown in 
Figure \ref{surg_diag_tildeA} (a), and we get Figure \ref{surg_diag_tildeA} (b). 
We slide the 2-handle with framing coefficient $\varepsilon$ another 2 times, 
and we get Figure \ref{surg_diag_tildeA} (c). 
We erase the canceling pair and get Figure \ref{surg_diag_tildeA} (d). 
We notice the knot in Figure \ref{surg_diag_tildeA} (d) is a ribbon knot. 
We denote this knot by $\tilde{K}_{(m,n)}$. 
By using the method in \cite{Ter59} again, we get 
\begin{align*}
\mathit{\Delta}_{\tilde{K}_{(m,n)}}(t) 
= -t^{|m|}+t^{|m|-1}-t+3-t^{-1}+t^{-|m|+1}-t^{-|m|} . 
\end{align*}
Then the Casson invariant of $\partial \tilde{A}(m,n-\frac{1}{2})$ 
can be computed by using the surgery formula for the Casson invariant as follows:
\begin{align*}
\lambda (\partial \tilde{A}(m,n-\frac{1}{2})) 
&= \lambda (S^3 + \frac{1}{\varepsilon}\cdot \tilde{K}_{(m,n)})\\
&= \lambda (S^3) + \frac{\varepsilon}{2}\mathit{\Delta}_{\tilde{K}_{(m,n)}}''(1)\\
&= 0 - \frac{\varepsilon}{2}\cdot 4|m|\\
&= 2m.
\end{align*}
Hence $\partial \tilde{A}(m,n-\frac{1}{2})$ is not homeomorphic to $S^3$ for 
$m\ne 0$, and 
$\tilde{A}(m,n-\frac{1}{2})$'s are mutually not homeomorphic for different $m$. 

\begin{figure}
\begin{center}
	\includegraphics[height=3.1cm]{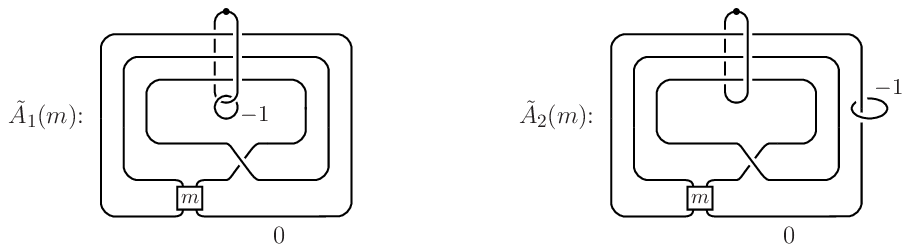}
	\caption{}
	\label{tildeA1andtildeA2}
\end{center}
\end{figure}

\vspace{1ex}
\noindent (2)
Set $n=-1$ and $m<0$. 
First we verify that $\tilde{A}(m,-\frac{3}{2})$ is a Stein domain. 
Deform the attaching circle of the 2-handle in the Kirby diagram of $\tilde{A}(m,-\frac{3}{2})$
obtained in Lemma \ref{lem:Kirby diag tildeA} to a Legendrian knot with respect to 
the canonical contact structure on $S^1\times S^2$ as shown in
Figure \ref{Legendred_tildeA} (in which we denote the 1-handle by 3-balls instead 
of dotted circle notation). Then we get its Thurston-Bennequin number as 
\[
tb = wr - \sharp \{ \text{left cusps}\} 
= (2|m| +1)- (2|m|-1)
=2. 
\]
Thus $\tilde{A}(m,-\frac{3}{2})$ is a Stein domain 
since the framing coefficient $0$ of the 2-handle is less than $tb=2$
(Eliashberg criterion \cite{Eli90,Gom98}). 

Next we prove that $(\tilde{A}(m,-\frac{3}{2}),\tilde{f}_m)$ is a cork. 
The basic idea of the proof is in \cite[Theorem 2.5]{AY08}. 

The involution $\tilde{f}_m$ extends to a self-homeomorphism of $\tilde{A}(m,-\frac{3}{2})$ 
since the boundary of $\tilde{A}(m,-\frac{3}{2})$ is a homology sphere 
and $B(0,m,n)$ is contractible 
(see \cite{Boy86} for a general discussion).

Let $\tilde{A}_1(m)$ and $\tilde{A}_2(m)$ be compact oriented 4-manifolds 
with boundary as shown in Figure \ref{tildeA1andtildeA2}. 
Note that 
$\tilde{A}_1(m)$ and $\tilde{A}_2(m)$ are obtained from $\tilde{A}(m,-\frac{3}{2})$ 
by attaching a 2-handle and $\tilde{A}_2(m)$ is diffeomorphic to the 4-manifold 
obtained from $\tilde{A}_1(m)$ by removing the interior of 
$\tilde{A}(m,-\frac{3}{2})$ and gluing it via $\tilde{f}_m$. 
Since $\tilde{f}_m$ extends to a self-homeomorphism of $\tilde{A}(m,-\frac{3}{2})$, 
$\tilde{A}_1(m)$ and $\tilde{A}_2(m)$ are homeomorphic. 
Now it is easy to check that $\tilde{A}_1(m)$ is a Stein domain 
(cf. Figure \ref{Legendred_tildeA}). 
On the other hand $\tilde{A}_2(m)$ is not a Stein domain since it has a 
2-sphere with self-intersection number $-1$ (cf. \cite{LM98}). 
Hence $\tilde{A}_1(m)$ and $\tilde{A}_2(m)$ are not diffeomorphic, that is, 
$\tilde{f}_m$ can not extend to a self-diffeomorphism of $\tilde{A}_1(m,-\frac{3}{2})$. 
\end{proof}

We recall that the \textit{shadow complexity} of a 4-manifold $M$ 
is the minimal number of true vertices among all shadows of $M$. 
If $M$ has 3- or 4-handles, 
its shadow complexity is defined by the minimal number of 
shadow complexities of 4-manifolds to which we can attach 3- and 4-handles 
such that the resulting manifold are diffeomorphic to $M$. 

\begin{proof}[Proof of Corollary \ref{cor:shadow complexity}]
By \cite{Ike71}, there are no acyclic special polyhedra without true vertices, 
and acyclic special polyhedra with only one ture vertex are just $A$ and $\tilde{A}$. 
A cork is defined as a Stein domain, whose handle decomposition 
has no 3-handles or 4-handles \cite{Eli90}. 
Thus there are no corks with shadow complexity $1$ other than 
$A(m,n)$ or $\tilde{A}(m,n-\frac{1}{2})$. 
\end{proof}

Now we compare our results and previous works. 
\begin{corollary}
\label{cor:W(l,k)}
For any integers $l$ and $k$, 
$W^{\pm}(l,k) \cong \tilde{A}(\pm 1, l+k-\frac{3}{2})$ holds.
\end{corollary}

\begin{proof}
Set $m=\pm 1$ and $n=k-1$ in the Kirby diagram of $\tilde{A}(m, n-\frac{1}{2})$
shown in Figure \ref{reconstruct_tildeA} (c), 
and we get $W^{\pm}(0,k) \cong \tilde{A}(\pm 1, k-\frac{3}{2})$. 
Akbulut and Kirby showed $W^{\pm} (l,k)\cong W^{\pm} (l+1,k-1)$ in \cite{AK79}, 
and then $W^{\pm} (l,k)\cong W^{\pm} (0,l+k) \cong \tilde{A}(\pm 1, l+k-\frac{3}{2})$. 
\end{proof}

\begin{remark}
If $(P,\gl )$ is a shadow of a 4-manifold $M$, 
then the 4-manifold constructed from the pair $(P,-\gl )$ is diffeomorphic to $-M$, 
where $-\gl$ is the gleam $\gl$ with the opposite sign. 
We apply this as follows: 
\begin{align*}
W^{-}(l,k)
\cong &\tilde{A}(-1, l+k-\frac{3}{2})\\
\cong &-\tilde{A}(1, -l-k+\frac{3}{2})\\
= &-\tilde{A}(1, (-l+2)+(-k+1)-\frac{3}{2})\\
\cong &-W^{+}(-l+2,-k+1). 
\end{align*}
Hence $W^{-}(l,k)\cong -W^{+}(-l+2,-k+1)$. 
Note that this assertion has been proven by 
Akbulut and Kirby in \cite{AK79}. 
Their proof is done by Kirby calculus. 
\end{remark}

The following is a corollary of Remark \ref{rem:A=tildeA}. 
\begin{corollary}
\label{cor:W(l,k)'}
For any integers $l$ and $k$, 
$W^{-}(l,k)\cong A(1,l+k-5)$ and $W^{+}(l,k)\cong A(-1,l+k+2)$ hold. 
\end{corollary}

\subsection{The case of $2$ true vertices}

Next we study the compact oriented 4-manifold $B(l,m,n)$ 
whose shadow is Bing's house $B$. 
We first describe $B(l,m,n)$ by the Kirby diagram and then prove 
Theorem \ref{thm:tildeA}. 

\begin{lemma}
\label{lem:Kirby diag B(l,m,n)}
The 4-manifold $B(l,m,n)$ is represented by the Kirby diagram
shown in Figure \ref{reconstruct_Bing} (c). 
\end{lemma}
\begin{proof}
The proof is similar to Lemma \ref{lem:Kirby diag A(m,n)} 
and \ref{lem:Kirby diag tildeA}. 
Let $P_{B}$ be a small regular neighborhood of the singular set 
of $B$ shown in Figure \ref{Nbd(Bing)}. 
It is easy to check that $B$ is homeomorphic to 
$P_{B}$ attached three disks along the boundary components 
$\partial (e_3\setminus P_{B}), \partial (e_4\setminus P_{B})$ and 
$\partial (e_5\setminus P_{B})$ shown in the figure. 

Now let us consider the torus $Q_{B}$ with two boundary components 
as shown in Figure \ref{2-punctured_torus}, 
where $\partial _4$ and $\partial _5$ are the boundary components, 
and $C_i$ ($i=3,4,5$) is the immersed curve on $Q_{B}$ 
equipped with over/under information at each double point. 

Let $Q'_{B}$ be the almost-special polyhedron obtained from $Q_{B}$ 
by attaching three disks $D_3, D_4$ and $D_5$ 
to $C_3,C_4$ and $C_5$ along the boundary respectively. 
It is easy to check that the almost-special polyhedron obtained from $Q'_{B}$ by collapsing 
the boundary region is $B$. 
Note that $D_i\subset e_i$ for $i=3,4,5$ under this identification. 

Next we consider the Kirby diagram shown in Figure \ref{reconstruct_Bing} (a). 
In the figure, $L_3,L_4$ and $L_5$ are attaching circles of 2-handles, 
whose framing coefficients are represented by 
$L, M$ and $N$ as in the figure respectively. 
We can see that $Q_B$ can be embedded in the Kirby diagram 
in Figure \ref{reconstruct_Bing} (a) such that 
$S^3\setminus (\text{dotted circles})$ collapses to $Q_B$ and 
$L_i$ is projected onto $C_i$ for $i=3,4,5$. 
We can also see that the 4-manifold given by 
Figure \ref{reconstruct_Bing} (a) has a shadow $Q'_B$, 
and then also $B$. 

\begin{figure}
\begin{tabular}{ccc}
\begin{minipage}{0.41\hsize}
	\begin{center}
	\includegraphics[width=4.9cm]{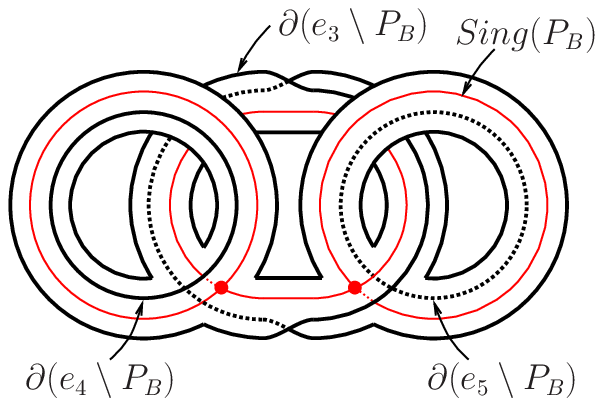}
	\caption{The almost-special polyhedron $P_B$.}
	\label{Nbd(Bing)}
	\end{center}
\end{minipage}
\begin{minipage}{0.05\hsize}
\hspace{1em}
\end{minipage}
\begin{minipage}{0.42\hsize}
\vspace*{3ex}
	\begin{center}
	\includegraphics[width=4.9cm]{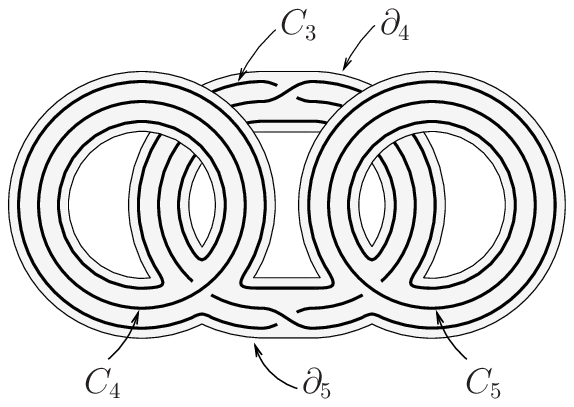}
	\caption{The torus $Q_B$ with two boundary components 
and immersed curves $C_3,C_4$ and $C_5$.}
	\label{2-punctured_torus}
	\end{center}
\end{minipage}
\end{tabular}
\begin{center}
	\includegraphics[width=1\hsize]{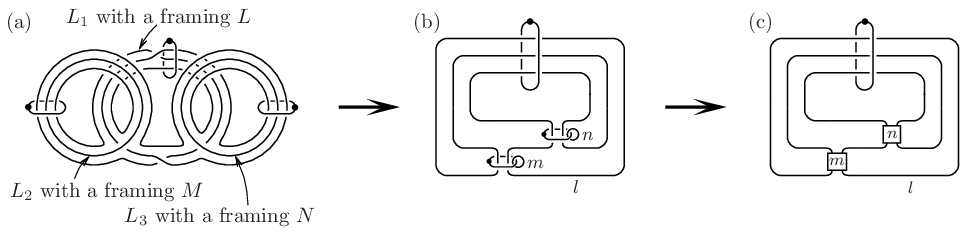}
	\caption{}
	\label{reconstruct_Bing}
\end{center}
\end{figure}

Let $F_i$ ($i=3,4,5$) be the framing with respect to $Q_{B}$ of ${C}_i$. 
Each ${F}_i$ is an annulus. 
We get the twist number $tw(C_i) = 0$ for each $i=3,4,5$. 
Therefore, $\gl ({D}_3) = L - 0$, $\gl ({D}_4) = M - 0$ and $\gl(D_5) = N - 0$ 
by Lemma \ref{lem:Kirby to shadow}. 
Considering the changes of the gleams by collapsing of the boundary regions, 
we get 
\begin{align*}
\gl (e_3) &= l = L + (4\cdot \frac{1}{2} - 4\cdot \frac{1}{2}), \\
\gl (e_4) &= m = M + 0, \\
\gl (e_5) &= n = N + 0. \\
\end{align*}
Hence $L=l, M=m$ and $N = n$. 
Thus we get the Kirby diagram of $B(l,m,n)$. 
By isotopy, we get Figure \ref{reconstruct_Bing} (b). 
We erase the canceling pairs and get Figure \ref{reconstruct_Bing} (c). 
\end{proof}

We now define 
\begin{align*}
f_{(m,n)}:\partial B(0,m,n)\to \partial B(0,m,n)
\end{align*}
by the involution obtained by first surgering $S^1 \times B^3$ to $B^2 \times S^2$ 
in the interior of $B(0,m,n)$, 
then surgering the other embedded $B^2 \times S^2$ back to $S^1 \times B^3$. 
We notice the Kirby diagram of $B(l,m,n)$ is a symmetric link, 
and the involution can be done by replacing the dot and g$0$h in the diagram. 

\begin{figure}
	\begin{center}
	\includegraphics[width=6cm]{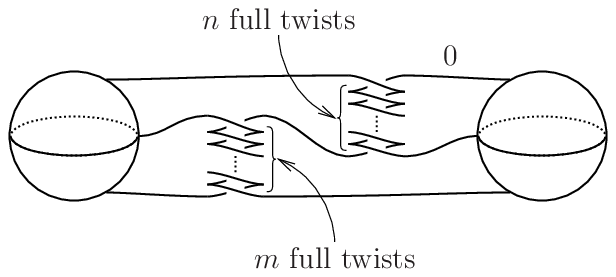}
	\caption{
}
	\label{Legendred_B(l,m,n)}
	\end{center}
\end{figure}

\begin{proof}[Proof of Theorem \ref{thm:B(l,m,n)}]
\noindent (1)
Set $|m|,|n|\geq 3$. 
Bing's house is a special polyhedron. 
The notation sl is introduced in \cite{IK14}, 
which represents the slope length of the Dehn filling 
corresponding to attaching a solid torus for each region of 
a special polyhedron. 
In our case, the values of sl on the three regions are
\begin{align*}
\text{sl} (e_1) &= \sqrt{4\gl (e_1)^2 + k(e_1)^2}= \sqrt{4l^2 + 100} ,\\
\text{sl} (e_2) &= \sqrt{4\gl (e_2)^2 + k(e_2)^2}= \sqrt{4m^2 + 1} ,\\
\text{sl} (e_3) &= \sqrt{4\gl (e_3)^2 + k(e_3)^2}= \sqrt{4n^2 + 1} . 
\end{align*}
Since these values are greater than 6, 
the boundary of $B(l,m,n)$ has a hyperbolic structure (cf. \cite{Lac00,CT08,IK14}). 
Since Bing's house is contractible then $B(l,m,n)$ is contractible. 
$B(l,m,n)$ satisfies the condition of the handle decomposition of
a Mazur manifold by the Kirby diagram shown in Figure \ref{reconstruct_Bing} (c). 
Hence $B(l,m,n)$ is Mazur type. 

\begin{figure}
\begin{center}
	\includegraphics[height=3.1cm]{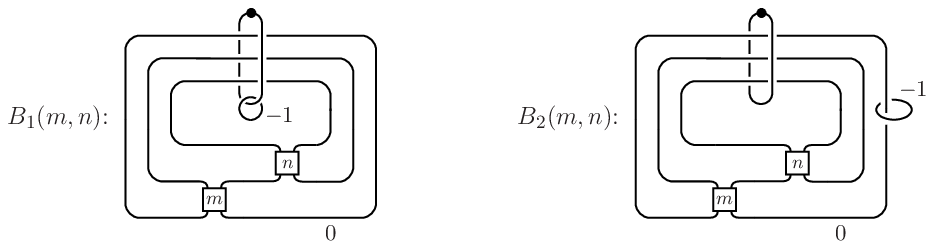}
	\caption{}
	\label{B1andB2}
\end{center}
\end{figure}

\vspace{1ex}
\noindent (2)
The proof is similar to the proof of Theorem \ref{thm:tildeA}. 
Set $l=0$ and $m,n<0$. 
Let us deform the attaching circle of $B(0,m,n)$ 
to a Legendrian knot with respect to the canonical contact structure 
on $S^1\times S^2$ as shown in Figure \ref{Legendred_B(l,m,n)} 
(in which we adopt the ball notation for the 1-handle). 
Then we can get the Thurston-Bennequin number 
\[
tb = wr - \sharp \{ \text{left cusps}\} 
= (2|m|+2|n|) - ((2|m|-1)+(2|n|-1))
=2. 
\]
Thus $B(0,m,n)$ is a Stein domain 
since the framing coefficient $0$ of the 2-handle is less than $tb=2$
(Eliashberg criterion \cite{Eli90, Gom98}). 

Next we prove that $(B(0,m,n),f_{(m,n)})$ is a cork. 
The basic idea of the proof is in \cite[Theorem 2.5.]{AY08}. 

$f_{(m,n)}$ extends to a self-homeomorphism of $B(0,m,n)$ 
since the boundary of $B(0,m,n)$ is a homology sphere 
and $B(0,m,n)$ is contractible 
(see \cite{Boy86} for a general discussion).

Let $B_1(m,n)$ and $B_2(m,n)$ be compact oriented 4-manifolds with boundary 
as shown in Figure \ref{B1andB2}. 
Note that 
$B_1(m,n)$ and $B_2(m,n)$ are obtained from $B(0,m,n)$ by attaching a 2-handle 
and $B_2(m,n)$ is diffeomorphic to the 4-manifold obtained from $B_1(m,n)$ 
by removing the interior of $B(0,m,n)$ and gluing it via $f_{(m,n)}$. 
Since $f_{(m,n)}$ extends to a self-homeomorphism of $B(0,m,n)$, 
$B_1(m,n)$ and $B_2(m,n)$ are homeomorphic. 
It is easy to check that $B_1(m,n)$ is a Stein domain 
(cf. Figure \ref{Legendred_B(l,m,n)}). 
On the other hand $B_2(m,n)$ is not a Stein domain since it has a 
2-sphere with self-intersection number $-1$ (cf. \cite{LM98}). 
Hence $B_1(m,n)$ and $B_2(m,n)$ are not diffeomorphic, that is, 
$f_{(m,n)}$ can not extend to a self-diffeomorphism of $B(0,m,n)$. 
\end{proof}

\begin{corollary}
\label{cor: Bing}
If $m$ and $n$ are the same sign, then $B(0,m,n)$ is Mazur type. 
\end{corollary}
\begin{proof}
Suppose $m,n<0$. 
If $\partial B(0,m,n) \cong S^3$, then $f_{(m,n)}$ is isotopic to the identity. 
Hence $B_1(m,n)$ and $B_2(m,n)$ are diffeomorphic, which is a contradiction. 
Therefore $B(0,m,n)$ is Mazur type. 

If $m,n>0$, since 
\begin{align*}
\partial B(0,m,n) \cong \partial (-B(0,-m,-n)) \cong -\partial (B(0,-m,-n)), 
\end{align*}
the assertion holds by the same argument. 
\end{proof}

\begin{remark}
The 4-manifold $B(0,-1,-1)$ is diffeomorphic to $\overline{W}_1$, 
which is a cork introduced by Akbulut and Yasui in \cite{AY08}. 
\end{remark}


\end{document}